\documentclass[a4paper,12pt]{article}
\usepackage{latexsym}	
\usepackage{amsmath}
\usepackage{amsthm}	

\usepackage{graphicx}
\usepackage{txfonts}
\usepackage{bm}
\usepackage{color}
\usepackage{epic,eepic}

\usepackage{geometry}
\numberwithin{equation}{section}
\newtheorem{theorem}{Theorem}[section]
\newtheorem{proposition}[theorem]{Proposition}
\newtheorem{remark}{Remark}[section]
\newtheorem{example}{Example}[section]
\newcommand{\OMIT}[1]{{\bf [OMIT:} #1 \ {\bf --- end OMIT] }}  %%% For work
   \renewcommand{\OMIT}[1]{}            %%% For FINAL
\newcommand{\RR}{{\mathbb{R}}}
\newcommand{\ZZ}{{\mathbb{Z}}}

\newcommand{\Rinf}{\RR \cup \{ +\infty \}}

\newcommand{\Rminf}{\RR \cup \{ -\infty \}}

\newcommand{\veczero}{{\bf 0}}
\newcommand{\suppp}{{\rm supp}\sp{+}}
\newcommand{\suppm}{{\rm supp}\sp{-}}
\newcommand{\unitvec}[1]{\bm{1}\sp{#1}}

\newcommand{\diag}{{\rm diag\,}}
\newcommand{\rank}{{\rm rank\,}}

\newcommand{\linl}{\mathrm{lineal}}
\newcommand{\lins}{\mathrm{linsp}}

\newcommand{\Lnat}{{L$^{\natural}$}}
\newcommand{\Mnat}{{M$^{\natural}$}}

\newcommand{\BnvexZb}{\mbox{\rm\bf (B$\sp{\natural}$-EXC[$\mathbb{Z}$])}}

\newcommand{\finbox}{\hspace*{\fill}$\rule{0.2cm}{0.2cm}$}
\newcommand{\todaye}{\the\year/\the\month/\the\day}

\begin{document}

\title{
Note on Minkowski Summation and Unimodularity \\
in Discrete Convex Analysis
}%title

\author{
Kazuo Murota%
\thanks{
The Institute of Statistical Mathematics,
Tokyo 190-8562, Japan; 
%% murota@ism.ac.jp; 
and
Faculty of Economics and Business Administration,
Tokyo Metropolitan University, 
Tokyo 192-0397, Japan,
murota@tmu.ac.jp}
%%(Corresponding author)}
\ and 
Akihisa Tamura%
\thanks{Department of Mathematics, Keio University, 
Yokohama 223-8522, Japan,
aki-tamura@math.keio.ac.jp}
}%%author

\date{December 2023 / April 2024}
%%\date{December 2023 / April 2024 (Version \today)}

\maketitle

\begin{abstract}
This short note gives an elementary alternative proof for  
a theorem of Danilov and Koshevoy 
on Minkowski summation and unimodularity
in discrete convex analysis.
It is intended to disseminate this fundamental theorem and
make its proof accessible to researchers in optimization
and operations research.
\end{abstract}

{\bf Keywords}:
Discrete optimization, 
discrete convex analysis,  
\Mnat-convex set,
Minkowski sum,
unimodular matrix

%%\subclass{52A41 \and 90C10}
%%52A41=Convex functions and convex programs
%%90C10=Integer programming
%%90C25 = convex programming
%%\memo{Keywords : 4 -- 6}

%%\tableofcontents

%\newpage

%%\input{unimoMinkwMainBody}

%% 2023-11-28 / 2023-12-04 / 2023-12-20 / 2024-03-02 / 2024-03-07 
%% 2024-04-12 / 2024-04-28 /
%% file = unimoMinkwMainBody.tex

\section{Introduction}
\label{SCintro}

Danilov--Koshevoy \cite{DK04uni}
made a fundamental contribution to 
our understanding of ``discrete convexity" in general.
For a class $\mathcal{C}$ 
of subsets of $\ZZ\sp{n}$,
they first identified desired properties that 
allow us to call $\mathcal{C}$ a class of discrete convex sets,
and clarified their mutual relations (often equivalences) among these properties.
The main result of \cite{DK04uni} is that 
a class subsets of $\ZZ\sp{n}$ that is qualified to be called ``discrete convexity"
can be constructed from a unimodular system, and vice versa.
As the name suggests, unimodular systems are closely related to unimodular matrices.
A survey paper of Koshevoy \cite{Kos11} describes subsequent development of this approach.

Minkowski summation is an intriguing operation in discrete setting.
This is fully recognized in discrete convex analysis \cite{Mdcasiam}
and indeed, the closedness under Minkowski summation is 
one of the desired properties of $\mathcal{C}$ discussed in \cite{DK04uni}.

For any sets $S_{1}$, $S_{2} \subseteq \RR\sp{n}$,
we denote their {\em Minkowski sum} (or {\em vector sum})
by $S_{1}+S_{2}$, that is,
\begin{equation*}  %% \label{minkowsumGdef}
S_{1}+S_{2} = \{ x + y \mid x \in S_{1}, \  y \in S_{2} \} .
\end{equation*}
If $S_{1}$ and $S_{2}$ are convex, then $S_{1}+S_{2}$ is also convex.
For any (possibly non-convex) sets $S_{1}$ and $S_{2}$ we have
\begin{equation} \label{minkowR}
\overline{S_{1}+ S_{2}} =  \overline{S_{1}} + \overline{S_{2}},
\end{equation}
where $\overline{S}$ denotes the convex hull of 
any set $S \subseteq \RR\sp{n}$.
In this connection it is emphasized that 
for discrete sets 
$S_{1}$, $S_{2} \subseteq \ZZ\sp{n}$, the relation
\begin{equation} \label{minkowNohole}
   S_{1}+S_{2}  = ( \overline{S_{1}+ S_{2}}) \cap \ZZ\sp{n} 
\end{equation}
does not always hold, even when  
$S_{i}  = \overline{S_{i}} \cap \ZZ\sp{n}$
for $i=1,2$.   The following example shows this.

\begin{example}[{\cite[Example 3.15]{Mdcasiam}}] \rm \label{EXicdim2sumhole}
The Minkowski sum of
$S_{1} = \{ (0,0), (1,1) \}$
and
$S_{2} = \{ (1,0),  \allowbreak   (0,1) \}$
is equal to 
$S_{1}+S_{2} = \{ (1,0), (0,1), (2,1), (1,2) \}$.
The relation \eqref{minkowNohole} fails,
since
$(1,1) \in (\overline{S_{1}+S_{2}}) \setminus (S_{1}+S_{2})$.
\finbox
\end{example}

Let $\mathcal{C}$ be a class of subsets of $\ZZ\sp{n}$
that has ``no-hole property," that is,
\begin{equation} \label{onesetNohole}
   S  =  \overline{S} \cap \ZZ\sp{n} 
\qquad (S \in \mathcal{C}).
\end{equation}
The following three categories of $\mathcal{C}$ can be distinguished
with respect to the Minkowski sum of two sets in $\mathcal{C}$.

\begin{enumerate}
\item
$S_{1}+S_{2}  \ne ( \overline{S_{1}+ S_{2}}) \cap \ZZ\sp{n}$
for some $S_{1}, S_{2} \in \mathcal{C}$.
See  Example~\ref{EXicdim2sumholeIC} below.

%% Integrally convex

\item
$S_{1}+S_{2}  = ( \overline{S_{1}+ S_{2}}) \cap \ZZ\sp{n}$ 
for all $S_{1}, S_{2} \in \mathcal{C}$,
but
$S_{1} + S_{2} \notin  \mathcal{C}$
for some $S_{1}, S_{2} \in  \mathcal{C}$.
See Example~\ref{EXlnatsetMsum} below.

%% \Lnat

\item
$S_{1} + S_{2} \in  \mathcal{C}$
for all $S_{1}, S_{2} \in \mathcal{C}$.
Then, by \eqref{onesetNohole},
$S_{1}+S_{2}  = ( \overline{S_{1}+ S_{2}}) \cap \ZZ\sp{n}$ 
is also true
for all $S_{1}, S_{2} \in \mathcal{C}$.

%% \Mnat
\end{enumerate}

In this paper we are interested in $\mathcal{C}$
in the third category, i.e., those 
classes $\mathcal{C}$ which are closed under Minkowski summation.
As is well known in discrete convex analysis,
the class of \Mnat-convex sets 
(see Section~\ref{SCpreMnatset} for the definition)
is such a class, and
by virtue of this property,
\Mnat-convexity finds applications in economics etc.
 \cite{DKM01,Mdcaeco16}.
An important finding of Danilov--Koshevoy \cite{DK04uni}
is that the class of \Mnat-convex sets is not the unique class in
the third category, but other classes can be constructed from
unimodular matrices.
The objective of this note is to give an elementary alternative proof for  
this theorem connecting Minkowski summation and unimodularity.
It is intended to disseminate this fundamental theorem 
among researchers in optimization and operations research,
together with a proof written in the language and style familiar to them.

\begin{example}\rm \label{EXicdim2sumholeIC}
The two sets $S_{1}$ and $S_{2}$ in Example \ref{EXicdim2sumhole}
are integrally convex sets; 
see \cite[Section~3.4]{Mdcasiam} 
for the definition of integral convexity.
This demonstrates that the class $\mathcal{C}$ of integrally convex sets
falls into the first category above.
\finbox
\end{example}

\begin{example}[{\cite[Example 3.11]{MS01rel}}] \rm \label{EXlnatsetMsum}
The Minkowski sum of
$S_{1} =  \{(0, 0, 0), (1, 1, 0)\}$
and
$S_{2} = \{(0, 0, 0), (0, 1, 1)\}$
is given by
$S_{1} + S_{2} = \{(0, 0, 0), (0, 1, 1), (1, 1, 0), (1, 2, 1)\}$.
The identity \eqref{minkowNohole} holds.
The two sets $S_{1}$ and $S_{2}$ are \Lnat-convex;
see \cite[Section~5.5]{Mdcasiam} for the definition of \Lnat-convex sets.
Their Minkowski sum $S_{1} + S_{2}$ is not \Lnat-convex,
since for
$x=(0, 0, 0), y=(1, 2, 1) \in S$,
we have
$\left\lceil (x+y)/2 \right\rceil = (1, 1, 1) \notin S$, 
$\left\lfloor (x+y)/2 \right\rfloor = (0, 1, 0) \notin S$,
a violation of the discrete midpoint convexity
(characterization of \Lnat-convexity).
It is known 
(\cite[Theorem~8.42]{Mdcasiam}, \cite[Proposition 3.14]{Msurvop21})
that the class $\mathcal{C}$ of \Lnat-convex sets
has the weaker property \eqref{minkowNohole}.
That is, the class of \Lnat-convex sets
falls into the second category above.
\finbox
\end{example}

Minkowski summation for other classes of discrete convex sets 
are summarized in Section~3.5 and Table 4 of 
a survey paper \cite{Msurvop21}
on operations on discrete convex sets and functions.

This paper is organized as follows. 
Section~\ref{SCprelim} is devoted to preliminaries 
on polyhedra, unimodular matrices, and \Mnat-convexity.
The theorem is described in Section~\ref{SCtheorem}
and the elementary proof is given in Section~\ref{SCproof}.

%%\newpage

\section{Preliminaries}
\label{SCprelim}

\subsection{Polyhedra}
\label{SCprePolyh}

%%% LLLLLL

A subset $P$ of $\RR\sp{n}$ is called a 
{\em polyhedron}
if it is described by a finite number of linear inequalities,
that is,
$P = \{ x \mid Cx \leq b \}$ for some matrix $C$ and a vector $b$.
A subset $Q$ of $\RR\sp{n}$ is called a {\em polytope}
if it is the convex hull of a finite number of points,
that is, $Q = \overline{S}$ for a finite subset $S$ of $\RR\sp{n}$,
where $\overline{S}$ denotes the convex hull of $S$. It is known 
that a polytope is nothing but a bounded polyhedron. 
The reader is referred to \cite{Gru03,Sch86,Zie07} for basic facts about polyhedra.

The {\em lineality space}
of a polyhedron $P$ $(\subseteq \RR\sp{n})$,
denoted by $\linl(P)$,
is defined as 
\begin{equation} \label{linldef}
 \linl(P) = \{ d \in \RR\sp{n} \mid x + \lambda d \in P 
 \ \mbox{ for all }\  x \in P, \lambda \in \RR \} ,
\end{equation}
which admits an alternative expression
\begin{equation} \label{linldef2}
 \linl(P) = \{ d \in \RR\sp{n} \mid x + \lambda d \in P 
 \ \mbox{ for all }\  \lambda \in \RR \} 
\end{equation}
with an arbitrarily fixed $x \in P$.
If $P=\{ x \in \RR\sp{n} \mid Cx  \le b \}$,
then 
$\linl(P)=\{ d \in \RR\sp{n} \mid Cd  = \veczero \}$.
We have $\linl(P) = \{ \veczero \}$
if and only if $P$ has a vertex. 
For each face $F$ of $P$, we have
\begin{equation} \label{linlPlinlF}
\linl(F)=\linl(P) .
\end{equation}

A polyhedron is said to be
{\em rational}
if it is described by a finite number of linear inequalities with
rational coefficients.
A polyhedron $P$ is an {\em integer polyhedron}
if $P=\overline{P \cap \ZZ\sp{n}}$, i.e., if
it coincides with the convex hull
of the integer points contained in it,
or equivalently, if
$P$ is rational and each face of $P$ contains an integer point.

\subsection{Unimodular matrices}
\label{SCpreMat}

A square matrix consisting of integer entries 
is called 
{\em unimodular}
if its determinant is $1$ or $-1$.
More generally,
an $n \times m$ matrix $A$ of integer entries
is called 
{\em unimodular}
if $\rank A = n$ (row-full rank)
and
each of its minor (subdeterminant) of order $n$ 
is an element in $\{ 0, +1, -1 \}$.
A matrix of integer entries
is called 
{\em totally unimodular}
if every minor (of any order)
is an element in $\{ 0, +1, -1 \}$.
When an $n \times m$ matrix $A$ is of the form
$A = [ I_{n} \mid C ]$,
where $I_{n}$ denotes the identity matrix of order $n$, 
 we have equivalences:
\[
\mbox{$A$ is unimodular $\Leftrightarrow$
$A$ is totally unimodular $\Leftrightarrow$
$C$ is totally unimodular.}
\]
Fundamental facts about (totally) unimodular matrices
are available in \cite[Chapters 19--21]{Sch86}.
It is particularly important to our paper that
a complete classification of totally unimodular matrices are known.
Network matrices form a major class of totally unimodular matrices,
whereas there exist totally unimodular matrices 
that are essentially different from network matrices.
As is fully recognized in optimization and operations research,
an integer optimization problems (IP) described by a unimodular matrix
admits an LP relaxation free from integrality gap.
In more general terms, unimodularity allows 
us to embed a discrete problem into a continuous setting.
Thus unimodularity is one of the key properties 
from structural and computational perspectives.

\subsection{\Mnat-convex polyhedra and \Mnat-convex sets}
\label{SCpreMnatset}

Let $N = \{ 1,2,\ldots, n \}$.
A g-polymatroid 
\cite{Fra11book,Fuj05book}
is a polyhedron $P$ described as
\begin{equation} \label{mnatsetineqR}
 P = \{ x \in \RR\sp{n} \mid 
  \mu(X) \leq x(X) \leq \rho(X)
   \ (X \subseteq N)  \},
\end{equation}
where $x(X) = \sum_{i \in X} x_{i}$, 
with a submodular function $\rho: 2\sp{N} \to \Rinf$
and a supermodular function $\mu: 2\sp{N} \to \Rminf$
satisfying the condition (called {\em paramodularity})
\begin{equation} \label{subsupmod}
 \rho(X) - \mu(Y) \geq  \rho(X \setminus Y) - \mu(Y \setminus X)  
\qquad (X, Y \subseteq N).
\end{equation}
It is assumed that
$\rho(\emptyset) = \mu(\emptyset) =0$,
and 
$\rho(N)$ and $\mu(N)$ are finite.

What is called an {\em \Mnat-convex polyhedron} 
in discrete convex analysis \cite{Mdcasiam} is in fact 
a synonym of a g-polymatroid.
In this paper we rely on the following characterization of an \Mnat-convex polyhedron,
where
$\unitvec{i}$ denotes the $i$th unit vector 
(characteristic vector of singleton set $\{ i \}$) for $i \in N$.

\begin{theorem}[{\cite[Theorem~17.1]{Fuj05book}}]  \label{THmnatedge}
A nonempty polyhedron $P$ in $\RR\sp{n}$ is a g-polymatroid (\Mnat-convex polyhedron) 
if and only if for every point $x$ in $P$ the tangent cone of $P$ at $x$
is generated by vectors chosen from 
$\unitvec{i}$, $-\unitvec{i}$
$(i \in N)$ and
$\unitvec{i} - \unitvec{j}$
$(i,j \in N)$.
\end{theorem}

The content of the above assertion is easier to understand 
for a bounded polyhedron (polytope).
In this special case the theorem states that a bounded polyhedron $P$ 
is an \Mnat-convex polyhedron if and only if
every edge of $P$ has a direction among the vectors 
$\unitvec{i}$, $-\unitvec{i}$
$(i \in N)$ and
$\unitvec{i} - \unitvec{j}$
$(i,j \in N)$.

The set of integer points in an integral g-polymatroid
coincides with what is called an {\em \Mnat-convex set}
in discrete convex analysis.
An \Mnat-convex set 
$S$ $(\subseteq \ZZ\sp{n})$
can be characterized by the following exchange property
\cite{Mdcasiam,MS99gp}:
\begin{description}
\item[\BnvexZb] 
For any $x, y \in S$ and $i \in \suppp(x-y)$, 
at least one of the following is true:
\\
(i) 
$x -\unitvec{i} \in S$ and $y+\unitvec{i} \in S$,
\\
(ii) 
there exists $j \in \suppm(x-y)$ such that 
$x-\unitvec{i}+\unitvec{j} \in S$ and $y+\unitvec{i}-\unitvec{j} \in S$,
\end{description}
where
\begin{equation*}  %% \label{vecsupportdef}
 \suppp(z) = \{ i \in N \mid z_{i} > 0 \},
\qquad 
 \suppm(z) = \{ j \in N \mid z_{j} < 0 \}
\end{equation*}
for any vector $z \in \ZZ\sp{n}$.

It is known \cite{Fuj05book,Mdcasiam} that 
the class of \Mnat-convex polyhedra
and the class of \Mnat-convex sets
are both closed under Minkowski summation.

\begin{theorem}  \label{THmnatconvol}
The Minkowski sum $P_{1} + P_{2}$ of \Mnat-convex polyhedra 
$P_{1}, P_{2}$ $(\subseteq \RR\sp{n})$ 
is an \Mnat-convex polyhedron,
and
the Minkowski sum $S_{1} + S_{2}$ of \Mnat-convex sets 
$S_{1}, S_{2}$ $(\subseteq \ZZ\sp{n})$ 
is an \Mnat-convex set.
\end{theorem}

%%\newpage

\section{Description of the Theorem}
\label{SCtheorem}

%%% DDDDD

We consider the following two conditions 
for a family $\mathcal{P}$ of polyhedra in $\RR\sp{n}$:
\begin{itemize}
\item
{\bf DCP1}:
Each $P \in \mathcal{P}$ is an integer polyhedron,

\item
{\bf DCP2}:
For any $P_{1}, P_{2} \in \mathcal{P}$, we have
$P_{1} + P_{2}  \in \mathcal{P}$ and
\begin{equation} \label{minkowNoholeDCP2}
(P_{1} + P_{2}) \cap \ZZ\sp{n} = (P_{1} \cap \ZZ\sp{n}) + (P_{2} \cap \ZZ\sp{n}),
\end{equation}
\end{itemize}
where DC stands for Discrete Convexity and  P for Polyhedron. 
For such $\mathcal{P}$, the corresponding family
of subsets of $\ZZ\sp{n}$, denoted by
\begin{equation} \label{famCfromFamP}
\mathcal{C} = \{ P \cap \ZZ\sp{n} \mid P \in \mathcal{P} \} ,
\end{equation}
satisfies the following two conditions:
\begin{itemize}
\item
{\bf DC1}:
For any $S \in \mathcal{C}$,
$\overline{S}$ is an integer polyhedron and
$S = \overline{S} \cap \ZZ\sp{n}$,

\item
{\bf DC2}:
For any $S_{1}, S_{2} \in \mathcal{C}$,
we have $S_{1} + S_{2} \in \mathcal{C}$.
\end{itemize}

\noindent
The condition DC2 says that $\mathcal{C}$ is closed under Minkowski summation.

Conversely, if a family $\mathcal{C}$ of subsets of $\ZZ\sp{n}$
satisfies DC1 and DC2, then
\begin{equation} \label{famPfromFamC}
\mathcal{P} = \{ \overline{S} \mid S \in \mathcal{C} \}
\end{equation}
satisfies DCP1 and DCP2.
This can be seen as follows.
For any $S \in \mathcal{C}$, its convex hull
$P=\overline{S}$ is an integer polyhedron by DC1, showing DCP1.
To show DCP2, let 
$P_{i} = \overline{S_{i}}$ for $S_{i} \in \mathcal{C}$ $(i=1,2)$.
We have
$S_{i} = P_{i} \cap \ZZ\sp{n}$ $(i=1,2)$
by DC1 and $S_{1}+ S_{2} \in \mathcal{C}$ by DC2.
On noting 
$P_{1} + P_{2}  
=\overline{S_{1}} + \overline{S_{2}} 
= \overline{S_{1}+ S_{2}}$
from \eqref{minkowR}, we obtain
$P_{1} + P_{2} = \overline{S_{1}+ S_{2}} \in \mathcal{P}$
and
$(P_{1} + P_{2}) \cap \ZZ\sp{n} 
= \overline{S_{1}+ S_{2}} \cap \ZZ\sp{n} = S_{1}+ S_{2}
=(P_{1} \cap \ZZ\sp{n}) + (P_{2} \cap \ZZ\sp{n})$.

In the following we construct a class $\mathcal{P}$ of polyhedra
with properties DCP1 and DCP2 from a unimodular matrix.
This in turn yields a class $\mathcal{C}$ of discrete sets satisfying DC1 and DC2.
To state the theorem we need some notations.
For any nonempty set $F$ $(\subseteq \RR\sp{n})$, we define
\begin{equation} \label{linspdef}
 \lins(F) = \{ \lambda (x-y) \in \RR\sp{n} \mid x, y \in F, \lambda \in \RR \},
\end{equation}
which is the subspace of $\RR\sp{n}$
parallel to the minimal affine space containing $F$.
Given an $n \times m$ unimodular matrix $A$ (with $\rank A = n$),
we denote by $\mathcal{P}_{A}$
the family of all integer polyhedra $P$ $(\subseteq \RR\sp{n})$
such that for any face $F$ of $P$, 
the subspace $\lins(F)$ is generated by some column vectors of the matrix $A$.
That is,
\begin{align}
\mathcal{P}_{A} =  \{ P \subseteq \RR\sp{n} \mid \ & 
 \mbox{$P$ is an integer polyhedron, and for any face $F$,}
\nonumber \\
& \mbox{$\lins(F)$ is generated by column vectors of $A$}
 \}.
\label{PAdef}
\end{align}
In this definition it is understood that
when $F$ is a vertex (zero-dimensional face),
$\lins(F)=\{ \veczero \}$
is generated by the empty subset of the column vectors of $A$.
The corresponding family 
$\mathcal{C}_{A}$
of discrete sets is defined as
\begin{equation}
\label{CAdef}
\mathcal{C}_{A} = \{ P \cap \ZZ\sp{n} \mid P \in \mathcal{P}_{A} \} .
\end{equation}

The following theorem is a formulation of
a result of Danilov--Koshevoy \cite{DK04uni}.
This theorem is expressed in the language familiar 
to researchers in optimization and operations research,
while the original paper \cite{DK04uni} employs abstract algebraic terms.

\begin{theorem}  \label{THunimosysDC}
For any $n \times m$ unimodular matrix $A$ (with $\rank A = n$),
the class $\mathcal{P}_{A}$ of polyhedra satisfies {\rm DCP1} and {\rm DCP2},
and the class $\mathcal{C}_{A}$ of discrete sets satisfies {\rm DC1} and {\rm DC2}. 
\end{theorem}
\begin{proof} 
As already mentioned,
the statement for $\mathcal{C}_{A}$ is equivalent to
that for $\mathcal{P}_{A}$.
The proof for $\mathcal{P}_{A}$ is given in Section~\ref{SCproof}.
\end{proof}

\begin{example} \rm \label{EXmconvedge}
By Theorem~\ref{THmnatedge},
an \Mnat-convex polyhedron
is characterized by vectors
$\unitvec{i}$, $-\unitvec{i}$
$(i \in N)$ and
$\unitvec{i} - \unitvec{j}$
$(i,j \in N)$
to generate tangent cones.
This means that the class of \Mnat-convex polyhedra
is given as 
$\mathcal{P}_{A}$ for
$A=A\sp{\hbox{\scriptsize M}}$
consisting of column vectors
$\unitvec{i}$
$(i \in N)$ and
$\unitvec{i} - \unitvec{j}$
$(i < j)$.
The matrix $A\sp{\hbox{\scriptsize M}}$ has 
$n$ rows and $n(n+1)/2$ columns.
When $n=4$, for example, we have
\begin{equation} \label{unimoMnatA}
A\sp{\hbox{\scriptsize M}} = 
\left[ \begin{array}{rrrr|rrrrrr}
1 & 0 & 0 & 0 & 1 & 1 & 1 & 0 & 0 & 0 
\\
0 & 1 & 0 & 0 & -1 & 0 & 0 & 1 & 1 & 0 
\\
0 & 0 & 1 & 0 & 0 & -1 & 0 & -1 & 0 & 1 
\\
0 & 0 & 0 & 1 & 0 & 0 & -1 & 0 & -1  & -1  
\\
\end{array}\right] .
\end{equation}
The matrix
$A\sp{\hbox{\scriptsize M}}$ is unimodular for general $n$,
and, by Theorem~\ref{THunimosysDC},
the class of \Mnat-convex polyhedra
is equipped with the properties DCP1 and DCP2.
Accordingly,
the class of \Mnat-convex sets,
$\mathcal{C}_{A}$ for $A=A\sp{\hbox{\scriptsize M}}$,
has the properties DC1 and DC2.
Thus Theorem~\ref{THunimosysDC} provides us with an alternative proof
of Theorem~\ref{THmnatconvol} that
the Minkowski sum of \Mnat-convex polyhedra (resp., sets) 
is an \Mnat-convex polyhedron (resp., set). 
\finbox
\end{example}

\begin{remark} \rm  \label{RMunimodMmaximal}
A unimodular matrix $A$ is said to be {\em maximal}
if it cannot be augmented by any nonzero column vector 
(other than $\pm a\sp{j}$
for a column vector $a\sp{j}$ of $A$)
without destroying the unimodularity.
The matrix $A\sp{\hbox{\scriptsize M}}$
for \Mnat-convex sets is maximal in this sense.
According to a result of 
Danilov--Koshevoy \cite{DK04uni}
showing a certain converse of Theorem~\ref{THunimosysDC},
the maximality of $A\sp{\hbox{\scriptsize M}}$ implies that 
there is no class of discrete convexity
that is a proper superclass of \Mnat-convex sets
and closed under Minkowski summation.
Maximal unimodular matrices are investigated in Danilov--Grishukhin \cite{DG99uni}.
\finbox
\end{remark}

\begin{remark} \rm  \label{RMunimodEquiv}
For an $n \times m$ unimodular matrix $A$
and an $n \times n$ square unimodular matrix $T$,
their product $A' =T A$ is an $n \times m$ unimodular matrix
and determines a class
$\mathcal{C}_{A'}$
of discrete convex sets.
This class
$\mathcal{C}_{A'}$
may be considered essentially equivalent to
$\mathcal{C}_{A}$,
since $S \in \mathcal{C}_{A}$ if and only if
$T S \in \mathcal{C}_{A'}$.
For example,
the matrix
$A'=T A\sp{\hbox{\scriptsize M}}$
obtained from
$A\sp{\hbox{\scriptsize M}}$
for \Mnat-convexity
with a diagonal matrix 
$T=\diag(1, \ldots, 1; -1, \ldots, -1 )$
corresponds to the so-called {\em twisted \Mnat-convexity},
which is used in the analysis of trading networks in economics
\cite{ISST15,IT15net,Mdcaeco16}.
There exist unimodular matrices that are not related to
$A\sp{\hbox{\scriptsize M}}$ by unimodular transformations.
Indeed, according to the classification of totally unimodular matrices mentioned 
in Section \ref{SCpreMat},
\begin{equation} \label{unimoD4}
B = 
\left[ \begin{array}{rrrr|rrrrr}
1 & 0 & 0 & 0 & 1 & 0 & 0 & 1 & 1  
\\
0 & 1 & 0 & 0 & 1 & 1 & 0 & 0 & 1  
\\
0 & 0 & 1 & 0 & 0 & 1 & 1 & 0 & 1  
\\
0 & 0 & 0 & 1 & 0 & 0 & 1 & 1 & 1   
\\
\end{array}\right]
\end{equation}
is a unimodular matrix that is independent of 
$A\sp{\hbox{\scriptsize M}}$,
and hence determines another class
$\mathcal{C}_{B}$ of discrete convex sets
satisfying  DC1 and DC2.
Such classes of discrete convex sets, different from \Mnat-convex sets,
are equally legitimate as discrete convexity,
but they await applications. 
\finbox
\end{remark}

\begin{remark} \rm  \label{RMotherproof}
Besides the original proof by Danilov--Koshevoy \cite{DK04uni},
alternative proofs for Theorem~\ref{THunimosysDC} can be found in
Baldwin--Klemperer \cite[Fact 4.9]{BK19},
Howard \cite[Theorem~4.5]{How07},
Tran--Yu \cite[Lemma 4]{TY19}.
These proofs employ algebra or tropical geometry.
\finbox
\end{remark}

\section{Proof of the Theorem}
\label{SCproof}

In this section we give an elementary proof of Theorem~\ref{THunimosysDC}.
We start with a general fact about the Minkowski sum 
of two polyhedra (independent of DCP1 and DCP2).
Recall notations
$\linl( \cdot )$ and $\lins( \cdot )$ from
\eqref{linldef} and \eqref{linspdef}, respectively.

\begin{proposition} \label{PRlinlPiFi}
Let $P_{1}$ and $P_{2}$
be polyhedra in $\RR\sp{n}$
and define
\begin{equation} \label{linlP1P2Ldef}
L :=  \linl(P_{1}) \cap  \linl(P_{2}).
\end{equation}
For any $z \in P_{1} + P_{2}$, there exist
$x \in P_{1}$ and
$y \in P_{2}$
satisfying 
$z = x + y$
and
\begin{equation} \label{linspL}
 \lins(F_{1}) \cap \lins(F_{2}) = L ,
\end{equation}
where 
$F_{1}$ denotes the minimal face of $P_{1}$ containing $x$
and 
$F_{2}$ is the minimal face of $P_{2}$ containing $y$.
\end{proposition}
\begin{proof}
Since $\linl(P_{i}) = \linl(F_{i})$
by \eqref{linlPlinlF}
and 
$\linl(F_{i}) \subseteq \lins(F_{i})$,
we have
\[
L =  \linl(P_{1}) \cap  \linl(P_{2}) = 
 \linl(F_{1}) \cap \linl(F_{2}) 
 \subseteq \lins(F_{1}) \cap \lins(F_{2}).
\]
We shall show that if 
$L \ne \lins(F_{1}) \cap \lins(F_{2})$,
we can modify
$x$ and $y$ so that
$\dim (\lins(F_{1})) + \dim(\lins(F_{2}))$ becomes smaller.

Suppose $L \ne \lins(F_{1}) \cap \lins(F_{2})$,
and take any
\[
d \in (\lins(F_{1}) \cap \lins(F_{2}))\setminus L .
\]
Obviously, $d \ne \veczero$. 
Since $F_{1}$ is the minimal face of $P_{1}$ containing $x$,
the point $x$ lies in the relative interior of $F_{1}$.
Similarly, $y$ is in the relative interior of $F_{2}$.
It then follows that 
for any $\varepsilon \in \RR$
with sufficiently small absolute value $|\varepsilon|$,
$x +\varepsilon d $ belongs to $F_{1}$ and 
$y - \varepsilon d$ belongs to $F_{2}$.
On the other hand, it follows from
$d \notin L =  \linl(F_{1}) \cap \linl(F_{2})$
that there exists some 
$\lambda \in \RR$
for which
$[ x + \lambda d \in \mathrm{bd}(F_{1}) \mbox{ {\small and} } y - \lambda d \in F_{2} ]$
or
$[ x + \lambda d \in F_{1} \mbox{ {\small and} } y - \lambda d \in \mathrm{bd}(F_{2}) ] $
is true,
where $\mathrm{bd}(F_{i})$ denotes the boundary of the face $F_{i}$
for $i=1,2$.
With the change of
$x \leftarrow x + \lambda d$ and $y \leftarrow y - \lambda d$,
we can make smaller at least one of the minimal faces $F_{1}$ and $F_{2}$,
while maintaining
$z = x + y$, $x \in P_{1}$, and $y \in P_{2}$.
Such changes decreases 
$\dim (\lins(F_{1})) + \dim(\lins(F_{2}))$ at least by one,
and eventually we arrive at \eqref{linspL}.
\end{proof}

\medskip

We are now ready to prove Theorem~\ref{THunimosysDC}.
To be specific, we prove that $\mathcal{P}_{A}$ satisfies DCP1 and DCP2
under the assumption that $A$ is a unimodular matrix.
DCP1 is obvious from the definition \eqref{PAdef} of $\mathcal{P}_{A}$.

To prove the first property
$P_{1} + P_{2}  \in \mathcal{P}_{A}$
in DCP2, 
we recall a well-known fact 
\cite[Section~1.5]{Gru03}
that
any face $F$ of $P_{1}+P_{2}$
can be represented as $F = F_{1}+F_{2}$
for some face $F_{1}$ of $P_{1}$
and some face $F_{2}$ of $P_{2}$.
For each $i=1,2$,
$\lins(F_{i})$ is generated by columns of $A$
since $P_{i} \in \mathcal{P}_{A}$.
Therefore,
\[
\lins(F) = \lins(F_{1}+F_{2}) = \lins(F_{1})+\lins(F_{2})
\]
is generated by columns of $A$.
Thus $P_{1} + P_{2}  \in \mathcal{P}_{A}$.

In the second property
\eqref{minkowNoholeDCP2}
in DCP2,
the inclusion $\supseteq$ is obvious.
For the reverse inclusion $\subseteq$
we will show that 
for any $z \in (P_{1} + P_{2}) \cap \ZZ\sp{n}$,
there exist
$x\sp{*} \in P_{1} \cap \ZZ\sp{n}$ and
$y\sp{*} \in P_{2} \cap \ZZ\sp{n}$
that satisfy $z = x\sp{*} + y\sp{*}$.
Take any $z \in (P_{1} + P_{2}) \cap \ZZ\sp{n}$.
By Proposition \ref{PRlinlPiFi}
there exist
$x \in P_{1}$ and
$y \in P_{2}$
satisfying 
\[
z = x + y, \qquad 
\lins(F_{1}) \cap \lins(F_{2}) = L ,
\]
where
$L =  \linl(P_{1}) \cap  \linl(P_{2})$,
$F_{1}$ is the minimal face of $P_{1}$ containing $x$, and
$F_{2}$ is the minimal face of $P_{2}$ containing $y$.

Since $P_{1}$ and $P_{2}$ are integer polyhedra,
there exist integer points 
$x\sp{\circ} \in F_{1}$ and $y\sp{\circ} \in F_{2}$.
Define
\begin{equation} \label{maru1}
 d_{z} := z - (x\sp{\circ} + y\sp{\circ}) = (x - x\sp{\circ}) +  (y - y\sp{\circ}) .
\end{equation}
Since
$z, x\sp{\circ}, y\sp{\circ} \in \ZZ\sp{n}$,
$x - x\sp{\circ} \in \lins(F_{1})$, and
$y - y\sp{\circ} \in \lins(F_{2})$, we have
\begin{equation} \label{maru2}
 d_{z} \in \ZZ\sp{n}
\quad \mbox{and} \quad  
 d_{z} \in \lins(F_{1}) + \lins(F_{2}) .
\end{equation}

We construct a basis of $\RR\sp{n}$ 
with reference to $\lins(F_{1})$ and $\lins(F_{2})$,
while making use of the fact that both $\lins(F_{1})$ and $\lins(F_{2})$
are generated by column vectors of $A$.
First we choose a basis of $\lins(F_{1})$
from the column vectors of $A$.
Let $\{ a\sp{1}, \ldots, a\sp{s} \}$
be the chosen columns vectors,
where the columns are re-numbered if necessary.
Next we add some other column vectors of $A$, say,
$\{ a\sp{s+1}, \ldots, a\sp{t} \}$
to form a basis of 
$\lins(F_{1}) + \lins(F_{2})$,
where we may assume that
$\{ a\sp{s+1}, \ldots, a\sp{t} \} \subseteq \lins(F_{2})$
since there exists a basis of $\lins(F_{2})$ consisting of column vectors of $A$.
Finally, we add some column vectors of $A$, say, 
$\{ a\sp{t+1}, \ldots, a\sp{n} \}$
to form a basis of $\RR\sp{n}$.
By construction we have
\begin{equation} \label{baseinlinsFi}
 \{ a\sp{1}, \ldots, a\sp{s} \} \subseteq \lins(F_{1}), \qquad
 \{ a\sp{s+1}, \ldots, a\sp{t} \} \subseteq \lins(F_{2}).
\end{equation}
\par
Express $d_{z}$ in \eqref{maru1} 
as a linear combination of the basis vectors, i.e., 
\begin{equation} \label{maru3}
 d_{z} = \sum_{i=1}\sp{n}  \lambda_{i} a\sp{i},
\end{equation}
where such representation is unique.
We have
\begin{equation} \label{maru4}
\lambda_{i} \in \ZZ \quad (i=1,\ldots,t), \qquad
\lambda_{i} =0 \quad (i=t+1,\ldots,n),
\end{equation}
because 
$d_{z} \in \ZZ\sp{n}$, the submatrix
$[ a\sp{1},  a\sp{2}, \ldots, a\sp{n} ]$
of $A$ formed by the basis vectors is unimodular,
$d_{z} \in \lins(F_{1}) + \lins(F_{2})$
as in \eqref{maru2},
and \eqref{baseinlinsFi} holds.
Define 
\begin{equation} \label{maru6}
 d_{x} := \sum_{i=1}\sp{s}  \lambda_{i} a\sp{i}, \quad
 d_{y} := \sum_{i=s+1}\sp{t}  \lambda_{i} a\sp{i}; \quad
 x\sp{*} := x\sp{\circ} + d_{x}, \quad
 y\sp{*} := y\sp{\circ} + d_{y}.
\end{equation}
Then we have
\begin{equation} \label{maru67}
d_{x} \in \lins(F_{1}) \cap \ZZ\sp{n}, \quad
d_{y} \in \lins(F_{2}) \cap \ZZ\sp{n}, \quad
x\sp{*} \in \ZZ\sp{n}, \quad
y\sp{*} \in \ZZ\sp{n},
\end{equation}
and
\begin{align} 
& d_{z} = d_{x} + d_{y},
\label{maru7dz}
\\ &
 z = x\sp{\circ} + y\sp{\circ} + d_{z} 
=  (x\sp{\circ} + d_{x}) + (y\sp{\circ} + d_{y})
= x\sp{*} + y\sp{*}  .
\label{maru7z}
\end{align}

Finally, we show that
$x\sp{*} \in P_{1}$ and $y\sp{*} \in P_{2}$.
It follows from
$x+y = z  = x\sp{*} + y\sp{*}$
that
\[
 x - x\sp{*} = y\sp{*} - y .
\]
We have $x - x\sp{*} \in  \lins(F_{1})$,
since
$x \in F_{1}$, $x\sp{*} = x\sp{\circ} + d_{x}$, $x\sp{\circ} \in F_{1}$,  
and $d_{x} \in \lins(F_{1})$.
Similarly, we have
$y\sp{*} - y \in  \lins(F_{2})$,
since
$y \in F_{2}$, $y\sp{*} = y\sp{\circ} + d_{y}$, $y\sp{\circ} \in F_{2}$,  
and $d_{y} \in \lins(F_{2})$.
Therefore,
\begin{equation} \label{maru678}
 x - x\sp{*} = y\sp{*} - y   \in  \lins(F_{1}) \cap \lins(F_{2}).
\end{equation}
Recalling \eqref{linlP1P2Ldef} and \eqref{linspL},
we obtain
\begin{align*}
 x - x\sp{*} = y\sp{*} - y   \in 
\lins(F_{1}) \cap \lins(F_{2}) 
= L = \linl(P_{1}) \cap \linl(P_{2}),
\end{align*}
from which follows that
\begin{equation} \label{maru9}
x\sp{*} = x  - (x - x\sp{*})  \in P_{1} + \linl(P_{1}) = P_{1},
\quad
y\sp{*} = y  + (y\sp{*} - y) \in P_{2} + \linl(P_{2}) = P_{2}.
\end{equation}

Combining \eqref{maru67}, \eqref{maru7z}, and \eqref{maru9},
we obtain
$z = x\sp{*} + y\sp{*}$,
$x\sp{*} \in P_{1} \cap \ZZ\sp{n}$, and
$y\sp{*} \in P_{2} \cap \ZZ\sp{n}$,
as desired.
This completes the proof of Theorem~\ref{THunimosysDC}.

%% end of file %%%%%%%%

\bigskip

\noindent {\bf Acknowledgement}.
%%The authors are grateful to the anonymous referees for helpful comments.
This work was supported by JSPS/MEXT KAKENHI JP23K11001 and JP21H04979, and
%%This work was supported 
by JST ERATO Grant Number JPMJER2301, Japan.

%%\newpage

%\input{unimoMinkwReflist}

%% 2023-11-28 / 2023-12-04 / 2024-04-26
%% file = unimoMinkwReflist.tex

%\newpage
%\tableofcontents
%%\listoffigures

\end{document}